\documentclass[11pt]{article}
\def\thetitle{Oriented discrepancy of Hamilton cycles}

\usepackage{graphicx}

\usepackage{amsmath,amssymb}

\usepackage[usenames,dvipsnames,svgnames,table]{xcolor}
\definecolor{CombinatoricaAqua}{HTML}{00698C}
\definecolor{CombinatoricaBlue}{HTML}{3A3293}
\definecolor{CombinatoricaBrown}{HTML}{66220C}
\definecolor{CombinatoricaRed}{HTML}{DF2A27}
\definecolor{HarvardCrimson}{rgb}{0.6471, 0.1098, 0.1882}

\makeatletter
\let\reftagform@=\tagform@
\def\tagform@#1{\maketag@@@
	{(\ignorespaces\textcolor{CombinatoricaBrown}{#1}\unskip\@@italiccorr)}}
\renewcommand{\eqref}[1]{\textup{\reftagform@{\ref{#1}}}}
\makeatother

\usepackage[backref=page]{hyperref}
\hypersetup{
	unicode,
	pdfencoding=auto,
	pdfauthor={Peleg Michaeli},
	pdftitle={\thetitle},
	pdfsubject={},
	pdfkeywords={},
	colorlinks=true,
	citecolor=CombinatoricaBlue,
	linkcolor=CombinatoricaAqua,
	anchorcolor=CombinatoricaBrown,
	urlcolor=HarvardCrimson}

\usepackage{amsthm}
\usepackage{thmtools}
\usepackage{bbm}
\usepackage{enumitem}
\usepackage[capitalize]{cleveref}
\Crefname{fact}{Fact}{Facts}
\Crefname{claim}{Claim}{Claims}


\declaretheoremstyle[
spaceabove=\topsep, spacebelow=\topsep,
headfont=\color{CombinatoricaBrown}\normalfont\bfseries,
bodyfont=\itshape,
]{thm}
\declaretheoremstyle[
spaceabove=\topsep, spacebelow=\topsep,
headfont=\color{CombinatoricaBrown}\normalfont\bfseries,
bodyfont=\normalfont,
]{dfn}
\declaretheoremstyle[
spaceabove=0.5\topsep, spacebelow=0.5\topsep,
headfont=\color{CombinatoricaBrown}\normalfont\bfseries,
bodyfont=\normalfont,
]{rmk}

\declaretheorem[style=thm,parent=section]{theorem}
\declaretheorem[style=thm,sibling=theorem]{lemma}
\declaretheorem[style=thm,sibling=theorem]{corollary}

\declaretheorem[style=thm,sibling=theorem]{observation}
\declaretheorem[style=thm,sibling=theorem]{conjecture}

\usepackage[nobysame,msc-links]{amsrefs}

\renewcommand{\eprint}[1]{\href{https://arxiv.org/abs/#1}{arXiv:#1}}

\BibSpec{book}{%
	+{}  {\PrintPrimary}                {transition}
	+{,} { \textbf}                     {title} 
	+{.} { }                            {part}
	+{:} { \textit}                     {subtitle}
	+{,} { \PrintEdition}               {edition}
	+{}  { \PrintEditorsB}              {editor}
	+{,} { \PrintTranslatorsC}          {translator}
	+{,} { \PrintContributions}         {contribution}
	+{,} { }                            {series}
	+{,} { \voltext}                    {volume}
	+{,} { }                            {publisher}
	+{,} { }                            {organization}
	+{,} { }                            {address}
	+{,} { \PrintDateB}                 {date}
	+{,} { }                            {status}
	+{}  { \parenthesize}               {language}
	+{}  { \PrintTranslation}           {translation}
	+{;} { \PrintReprint}               {reprint}
	+{.} { }                            {note}
	+{.} {}                             {transition}
	+{}  {\SentenceSpace \PrintReviews} {review}
}
\BibSpec{incollection}{%
  +{}  {\PrintAuthors}                {author}
  +{,} { \textit}                     {title}
  +{.} { }                            {part}
  +{:} { \textit}                     {subtitle}
  +{,} { \PrintContributions}         {contribution}
  +{,} { \PrintConference}            {conference}
  +{}  {\PrintBook}                   {book}
  +{,} { }                            {booktitle}
	+{}  { \PrintEditorsB}              {editor}
	+{,} { }                            {publisher}
  +{,} { \PrintDateB}                 {date}
  +{,} { pp.~}                        {pages}
  +{,} { }                            {status}
  +{,} { \PrintDOI}                   {doi}
  +{,} { available at \eprint}        {eprint}
  +{}  { \parenthesize}               {language}
  +{}  { \PrintTranslation}           {translation}
  +{;} { \PrintReprint}               {reprint}
  +{.} { }                            {note}
  +{.} {}                             {transition}
  +{}  {\SentenceSpace \PrintReviews} {review}
}

\makeatletter
\renewcommand{\PrintNames@a}[4]{%
	\PrintSeries{\name}
	{#1}
	{}{ and \set@othername}
	{,}{ \set@othername}
	{}{ and \set@othername}
	{#2}{#4}{#3}%
}
\makeatother

\makeatletter
\def\mathcolor#1#{\@mathcolor{#1}}
\def\@mathcolor#1#2#3{%
	\protect\leavevmode
	\begingroup
	\color#1{#2}#3%
	\endgroup
}
\makeatother
\definecolor{Red}{rgb}{0.618,0,0}
\definecolor{Blue}{rgb}{0,0,1}
\definecolor{Green}{rgb}{0,0.298,0}

\usepackage{sectsty}
\sectionfont{\color{CombinatoricaBrown}}
\subsectionfont{\color{CombinatoricaBrown}}
\subsubsectionfont{\color{CombinatoricaBrown}}

\usepackage{soul}
\soulregister\ref7

\usepackage{pifont}
\usepackage{calc}

\usepackage[a4paper]{geometry}
\title{\thetitle}
\author{
  Lior Gishboliner\thanks{
    Department of Mathematics, ETH Z\"{u}rich, Switzerland.
    Email: \href{mailto:lior.gishboliner@math.ethz.ch}
                {\tt lior.gishboliner@math.ethz.ch}. Research supported by SNSF grant 200021\_196965.}
  \and
  Michael Krivelevich\thanks{
    School of Mathematical Sciences,
    Tel Aviv University,
    Tel Aviv 6997801, Israel.
    Email: \href{mailto:krivelev@tauex.tau.ac.il}
                {\tt krivelev@tauex.tau.ac.il}.
    Research supported in part by USA--Israel BSF grant 2018267.}
  \and
  Peleg Michaeli\thanks{
    Department of Mathematical Sciences,
    Mellon College of Science,
    Carnegie Mellon University,
    Pittsburgh, PA, USA.
    Email: \href{mailto:pelegm@cmu.edu}
                {\tt pelegm@cmu.edu}.}
}

\makeatletter
\def\namedlabel#1#2{\begingroup
  #2%
  \def\@currentlabel{#2}%
  \phantomsection\label{#1}\endgroup
}
\makeatother

\usepackage[title]{appendix}

\usepackage{mleftright}
\mleftright

\newcommand{\defn}[1]{{\bfseries #1}}

\newcommand{\eps}{\varepsilon}
\renewcommand{\phi}{\varphi}

\newcommand{\sm}{\smallsetminus}
\newcommand{\es}{\varnothing}

\newcommand{\pr}[0]{\mathbb{P}}

\usepackage{comment}

\usepackage{tabto}

\newcommand{\whp}[0]{\textbf{whp}}

\usepackage{subcaption}
\usepackage{tikz}
\usetikzlibrary{calc}

\begin{document}
\maketitle

\begin{abstract}
  We propose the following extension of Dirac's theorem:
  if $G$ is a graph with $n\ge 3$ vertices and minimum degree $\delta(G)\ge n/2$,
  then in every orientation of $G$
  there is a Hamilton cycle with at least $\delta(G)$ edges oriented in the same direction.
  We prove an approximate version of this conjecture,
  showing that minimum degree $\frac{n+8k}{2}$ guarantees a Hamilton cycle with at least $(n+k)/2$ edges oriented in the same direction.
  We also study the analogous problem for random graphs,
  showing that if the edge probability $p = p(n)$ is above the Hamiltonicity threshold,
  then, with high probability, in every orientation of $G \sim G(n,p)$ there is a Hamilton cycle with $(1-o(1))n$ edges oriented in the same direction. 
\end{abstract}

\section{Introduction}
Recently there has been much interest in discrepancy-type problems on graphs.
In the general setting of combinatorial discrepancy theory, one is given a hypergraph $\mathcal{H}$ and should colour the vertices of $\mathcal{H}$ with two colours so that each edge is as balanced as possible.
More precisely, given a colouring $f : V(\mathcal{H}) \rightarrow \{-1,1\}$, define the imbalance of an edge $e \in \mathcal{H}$ to be $D(e) = \left| \sum_{x \in e}{f(x)} \right|$. The discrepancy of $\mathcal{H}$ is then defined as 
$\min_{f : V(\mathcal{H}) \rightarrow \{-1,1\}} \max_{e \in E(\mathcal{H})}D(e)$. 
In other words, the discrepancy is the maximum {\em guaranteed imbalance}. 
We refer the reader to Chapter~4 in the book of Matou\v{s}ek~\cite{Mat} for an overview of combinatorial discrepancy. 

Recent works considered instances of the above general setting, in which the hypergraph $\mathcal{H}$ originates from a graph, that is, $V(\mathcal{H})$ is the edge-set of some graph $G$, and $E(\mathcal{H})$ is a family of certain subgraphs of $G$. The goal is to identify conditions on $G$ which guarantee that there exists a high-discrepancy subgraph of the specific type.
For example, Balogh, Csaba, Jing and Pluh\'{a}r~\cite{BCJP20} proved that if $G$ has $n$ vertices and minimum degree $(3/4 + \eps)n$, then in every $2$-colouring of $E(G)$, there is a Hamilton cycle with at least $(1/2+\varepsilon/64)n$ edges of the same colour (and hence discrepancy at least $\varepsilon n/32$). This was generalised to an arbitrary number of colours (with a suitable definition of multicolour discrepancy) in~\cite{GKM22} and independently~\cite{FHLT21}. This result (for two colours) was extended to $K_r$-factors and $r$-powers of Hamilton cycles in~\cite{BCPT21} and~\cite{Bra21+}, respectively. In~\cite{GKM22+}, the authors determined the multicolour discrepancy of Hamilton cycles in the complete graph, and showed that if $p$ is above the Hamiltonicity threshold, then the discrepancy of Hamilton cycles in $G(n,p)$ is essentially as large as in the complete graph.  

In the present work, we consider an {\em oriented} version of graph discrepancy.
Here, instead of colouring the edges of the graph with two colours,
one considers orientations of the edges. Throughout this paper, we use $\{u,v\}$ to denote an unordered edge (or pair), and $(u,v)$ to denote the directed edge from $u$ to $v$. 
Hamilton cycles are naturally suitable for the oriented setting since they have a notion of direction. 
The goal is thus to find a Hamilton cycle with as many edges as possible which are oriented in the direction of the cycle.
More precisely, for an $n$-vertex graph $G$, an orientation $D$ of $G$ and a Hamilton cycle $C = v_1,\dots,v_n,v_1$ in $G$, we say that the edge $\{v_i,v_{i+1}\}$ (with indices taken modulo $n$) is a {\em forward edge of $C$} if $(v_i,v_{i+1}) \in E(D)$. Note that we always fix one of the two directions of the Hamilton cycle $C$. 
Let $\vec{e}_D(C)$ be the number of forward edges. The {\em oriented discrepancy of Hamilton cycles} in $G$ is the maximum $k$ such that for every orientation $D$ of $G$, there is a Hamilton cycle $C$ in $G$ such that $\vec{e}_D(C) \geq \frac{n+k}{2}$. 
Generalising Dirac's theorem~\cite{Dir52}, we propose the conjecture that if $\delta(G) \geq \frac{n+k}{2}$ then the oriented discrepancy of Hamilton cycles in $G$ is at least $k$. It is convenient to state this conjecture in the following \nolinebreak form:
\begin{conjecture}\label{conj:Dirac}
  Let $n\ge 3$
  and let $G$ be an $n$-vertex graph with $\delta(G)\ge n/2$.
  Then, in any orientation of the edges of $G$,
  there exists a Hamilton cycle in which at least $\delta(G)$ of the edges are pointing forward.
\end{conjecture}

If true, \cref{conj:Dirac} would be best possible.
To see this, fix $n\ge 3$ and $\delta\ge n/2$,
and consider the graph $G$ with vertex-partition
$V(G) = A \cup B$, where $|A| = n-\delta$ and $|B| = \delta$.
The edge set is the set of all pairs which touch $B$.
It is easy to see that $\delta(G) = \delta$.
Now, orient all edges from $A$ to $B$ (and orient the edges inside $B$ arbitrarily). 
Consider any Hamilton cycle in $G$, and fix a cyclic direction.
Any edge pointing forward (with respect to that direction) ends at a distinct vertex of $B$;
hence there are at most $|B|=\delta$ such edges. We note that \cref{conj:Dirac} holds when $G$ is a complete graph due to the fact that every tournament has a directed Hamilton path.

Our first result is an approximate version of \cref{conj:Dirac}.
\begin{theorem}\label{thm:Dirac}
  Let $k \geq 0$
  and let $G$ be an $n$-vertex graph with $n \geq 30 + 4(k-1)$ and $\delta(G)\ge \frac{n + 8k}{2}$.
  Then in every orientation of $G$, there is a Hamilton cycle with at least $\frac{n+k}{2}$ edges in the same direction.  
\end{theorem}

As mentioned above, \cref{conj:Dirac} can be viewed as an extension of Dirac's theorem, stating that Dirac's theorem holds in a ``robust way". 
Namely, not only do graphs with large minimum degree (above $n/2$) have Hamilton cycles, but the collection of Hamilton cycles is robust enough to guarantee the existence of an unbalanced Hamilton cycle in every edge-orientation. 
It is worth mentioning that there is by now an extremely rich literature (certainly too rich to be comprehensively surveyed here) on extensions and generalisations of Dirac's theorem in various directions.
We refer the reader to the surveys of Gould~\cites{Gou91,Gou02,Gou14} and of K\"uhn and Osthus~\cites{KO12,KO14}, as well as the survey of Sudakov~\cite{Sud17} focusing on robustness-type results. 

~

As our second result, we study oriented discrepancy of Hamilton cycles in random graphs.
We show that if the edge probability $p$ is above the Hamiltonicity threshold,
then with high probability\footnote{i.e., with probability tending to $1$ is $n$ tends to infinity} (\whp{} for short), in every orientation of $G(n,p)$ there is a Hamilton cycle with {\em almost all} edges oriented in the same direction.
\begin{theorem}\label{thm:Gnp}
  Let $n$ be an integer and let $p\ge (\log{n}+\log{\log{n}}+\omega(1))/n$.
  Then $G\sim G(n,p)$ is \whp{} such that
  in any orientation of its edges
  there exists a Hamilton cycle with at least $(1-o(1))n$ edges in the same direction.
\end{theorem}

The proof of \cref{thm:Gnp} combines machinery from our earlier paper~\cite{GKM22+}
with a result on the oriented discrepancy of almost-spanning paths in sparse expanders (see \cref{thm:beta:path}).
In the special case of random graphs, this result is as follows.
\begin{theorem}\label{thm:Gnp:path}
  For every $\delta > 0$ there is $C > 0$ such that $G \sim G(n,C/n)$ \whp{} satisfies the following:
  in any orientation of the edges of $G$
  there exists a path of length at least $(1-\delta)n$ with at most $\delta n$ edges going backwards.
\end{theorem}

\Cref{thm:Gnp:path} has the following interesting implication to unbalanced Hamilton cycles in tournaments.
It is a well-known fact \cite{Redei} that every tournament has a directed Hamilton path. Hence, in any tournament, there is a Hamilton cycle (of the underlying complete graph) in which all edges but at most one go in the same direction.
It may be, however, that a tournament has only one directed Hamilton path (this is the case in the transitive tournament). It is then natural to relax the notion of ``imbalance'' by considering Hamilton cycles which have up to $\delta n$ edges in the opposite direction, and ask how many such unbalanced Hamilton cycles are guaranteed to exist. The following corollary addresses this question.

\begin{corollary}\label{cor:multiplicity}
  For every $\delta > 0$ there is $C > 0$ such that every tournament $T$ on $n$ vertices contains at least $C^{-n}n^n$ Hamilton cycles
  (of the underlying complete graph), each of which has at least $(1 - \delta)n$ edges in the same direction.
\end{corollary}

\Cref{cor:multiplicity} is optimal up to the value of the constant $C$. Indeed, let $T$ be the random tournament where each edge is oriented randomly (with probability $1/2$ for each direction). For every Hamilton cycle $C$ in the underlying complete graph of $T$, the probability that $C$ has at least $(1/2 + \varepsilon)n$ forward edges is at most $e^{-\Omega(\varepsilon^2 n)}$, by the Chernoff bound. Hence, the expected number of Hamilton cycles having at least $0.51n$ (say) forward edges, is at most $e^{-\Omega(n)}n!$. By a similar consideration, the expected number of Hamilton cycles having at least $(1-\delta)n$ edges in the same direction is at most $c^{-n} n!$ where $c = 2^{1 - O(\delta\log(1/\delta))}$.
Determining the correct dependence of $C$ on $\delta$ in \cref{cor:multiplicity} would be interesting. 

~

\Cref{thm:Dirac} is proven in the next section,
\cref{thm:Gnp,thm:Gnp:path} are proven in \cref{sec:random}
and \cref{cor:multiplicity} is proven in \cref{sec:tournament}.

\section{Unbalanced Hamilton cycles in Dirac graphs}
A \defn{diamond} is the graph with vertices $a,b,c,d$ and edges $\{a,b\},\{a,c\},\{a,d\},\{b,c\},\{b,d\}$.
We call $c,d$ the \defn{ends} of the diamond, and $a,b$ the \defn{centre} of the diamond.
Given an orientation of its edges,
a diamond is called \defn{good} (with respect to that orientation) if $c$ and $d$ relate in the same way to $a,b$.
(The direction of the edge $\{a,b\}$ will not be important, as the roles of $a,b$ are the same.
For convenience, we assume that $a \rightarrow b$.)
There are four good diamonds, depicted in \cref{fig:diamonds}.

\begin{figure*}[t!]
  \captionsetup{width=0.879\textwidth,font=small}
  \centering
  \includegraphics{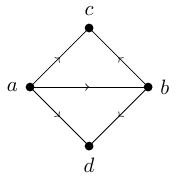}
  \hspace{0.5cm}
  \includegraphics{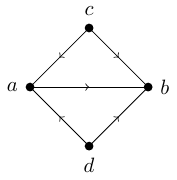}
  \hspace{0.5cm}
  \includegraphics{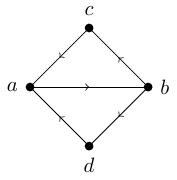}
  \hspace{0.5cm}
  \includegraphics{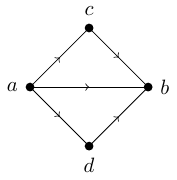}
  \caption{The four good diamonds.}
  \label{fig:diamonds}
\end{figure*}

We say that a path is \defn{positive} (with respect to an orientation) if it consists of more forward edges than backward edges when traversing the path from its first vertex to its last.
Similarly, it is \defn{negative} if it consists of more backward edges than forward edges.
The following is the key property of good diamonds that we will use. It can easily be verified by hand (see \cref{fig:diamonds}).
\begin{observation}\label{obs:good diamond paths}
Let $D$ be a good diamond with ends $c,d$. Then $D$ contains a positive path of length 3 from $c$ to $d$, as well as a positive path of length 3 from $d$ to $c$.
\end{observation}

\begin{lemma}\label{lem:one good diamond}
  Let $G$ be a graph with $|V(G)| = n \geq 30$ vertices $e(G) \geq \lfloor \frac{n^2}{4} \rfloor +1$.
  Then every orientation of $G$ contains a good diamond. 
\end{lemma}
\begin{proof}
Edwards~\cite{Edw77} and independently Khad\v{z}iivanov and Nikiforov~\cite{KN79} proved that a graph with $n$ vertices and $\lfloor \frac{n^2}{4} \rfloor +1$ edges has an edge which is contained in at least $n/6$ triangles (earlier, Erd\H{o}s~\cite{Erd62} proved a bound of $n/6 - O(1)$).
So let $\{a,b\} \in E(G)$ be an edge contained in at least $n/6$ triangles,
and let $x_1,\dots,x_m$,
$m \geq n/6 \geq 5$ be such that $a,b,x_i$ is a triangle in $G$.
Fix any orientation of $G$.
By the pigeonhole principle,
there are $1 \leq i < j \leq m$ such that $x_i,x_j$ relate in the same way (in the orientation) to $a,b$.
Then $a$, $b$, $c := x_i$ and $d := x_j$ form a good diamond in $G$,
as required.
\end{proof}
\begin{lemma}\label{lem:vertex-disjoint good diamonds}
  Let $k \geq 1$, and let $G$ be a graph with $|V(G)| = n \geq 30 + 4(k-1)$ and $e(G) \geq \frac{n^2}{4} + 2(k-1)n - 4k^2 + 6k - 1$.
  Then every orientation of $G$ contains $k$ vertex-disjoint good diamonds. 
\end{lemma}
\begin{proof}
The proof is by induction on $k$.
Fix any orientation of $G$. It is easy to check that $e(G) \geq n^2/4 + 1$ by the assumptions of the lemma. 
By \cref{lem:one good diamond}, $G$ contains a good diamond $D$ (this also establishes the base case $k = 1$).
Let $G'$ be the subgraph of $G$ induced by $V(G) \sm V(D)$.
We have $|V(G')| = n - 4 \geq 30 + 4(k-2)$ and 
\[
  \begin{aligned}
    e(G')
    &\geq e(G) - 6 - 4(n-4)
    = e(G) - (4n-10)
    \geq \frac{n^2}{4} + 2(k-1)n - 4k^2 + 6k - 1 - (4n-10)\\
    &= \frac{(n-4)^2}{4} + 2(k-2)(n-4) - 4(k-1)^2 + 6(k-1) - 1.
  \end{aligned}
\]
By the induction hypothesis, $G'$ contains $k-1$ vertex-disjoint good diamonds, finishing the \nolinebreak proof.
\end{proof}

We remark that \cref{lem:vertex-disjoint good diamonds} is optimal in the following sense.
Consider the complete $n$-vertex tripartite graph $G=(A\cup B\cup C,E)$ where $|A|=k-1$ and $|B|,|C|$ are nearly equal.
It is not hard to check that $G$ has $n^2/4+\Theta(kn)$ edges but at most $k-1$ vertex-disjoint diamonds.

The last ingredient in the proof of Theorem \ref{thm:Dirac} is the following lemma due to P\'osa. 
\begin{lemma}[\cite{Pos63}]\label{lem:path_extension_to_Hamilton_cycle}
	Let $t \geq 0$ and let $G$ be a graph with $n$ vertices and minimum degree at least $\frac{n+t}{2}$. Let $E \subseteq E(G)$ be the edge-set of a path forest\footnote{A path forest is a forest in which every connected component is a path.} of size at most $t$. Then there exists a Hamilton cycle in $G$ which uses all edges in $E$.
\end{lemma}

\begin{proof}[Proof of \cref{thm:Dirac}]
If $k = 0$ then the assertion follows immediately from Dirac's theorem. So suppose from now on that $k \geq 1$.
Fix any orientation of $G$.
We have $e(G) \geq \frac{n}{2} \cdot \frac{n+8k}{2} = \frac{n^2}{4} + 2kn \geq \frac{n^2}{4} + 2(k-1)n - 4k^2 + 6k - 1$,
hence we may apply \cref{lem:vertex-disjoint good diamonds}.
 
Let $D_1,\dots,D_k$ be vertex-disjoint good diamonds in $G$.
Let us denote by $c_i,d_i$ the ends of $D_i$.
For each $1 \leq i \leq k$,
let $P_i \subseteq D_i$ be a positive path of length 3 from $c_i$ to $d_i$,
and let $Q_i \subseteq D_i$ be a positive path of length 3 from $d_i$ to $c_i$; such paths exists by \cref{obs:good diamond paths}.
Let $E = \bigcup_{i = 1}^{k}{E(P_i)}$.
So $E$ is the edge-set of a path forest,
and $|E| = 3k$.
By \cref{lem:path_extension_to_Hamilton_cycle},
there is a Hamilton cycle $H$ of $G$ which contains all edges in $E$; in other words,
$H$ contains $P_1,\dots,P_k$ as subpaths.
Fix a direction of $H$ which matches the orientation of the majority of the edges in $E(H) \sm E$.
We now define a Hamilton cycle $H'$ as follows.

For $i = 1,\dots,k$,
we say that $D_i$ {\em matches the direction of $H$} if,
when traversing $P_i \subseteq H$ along the fixed direction of $H$,
we go from $c_i$ to $d_i$ (and not from $d_i$ to $c_i$).
Recall that $P_i$ is positive when traversing it from $c_i$ to $d_i$ (and hence negative when traversing it from $d_i$ to $c_i$).
Let $I$ be the set of $i = 1,\dots,k$ such that $P_i$ does {\em not} match the direction of $H$.
Now let $H'$ be the Hamilton cycle with edge-set $E(H') = 
\left( E(H) \sm \bigcup_{i \in I}{E(P_i)} \right) \cup \bigcup_{i \in I}{E(Q_i)}$.
In other words,
for each $i \in I$,
we replace $P_i$ with $Q_i$.
It is easy to see that $H'$ is indeed a Hamilton cycle,
and that it has at least $\frac{n+k}{2}$ edges in the same direction.
\end{proof}

\section{Unbalanced Hamilton cycles in sparse expanders}
\label{sec:random}
To prove \cref{thm:Gnp} we first prove \cref{thm:beta:path} below,
which extends \cref{thm:Gnp:path} to the more general setting of graphs with certain expansion properties. More precisely, for $\beta > 0$, a graph $G$ is called a {\em $\beta$-graph} if there is an edge in $G$ between every pair of disjoint $U,W \subseteq V(G)$ with $|U|,|W| \geq \beta |V(G)|$. This notion was introduced in \cite{FK21}.
\begin{theorem}\label{thm:beta:path}
  For every $\delta > 0$ there exists $\beta>0$
  such that the following holds.
  Let $G$ be a $\beta$-graph on $n$ vertices. 
  Then, in any orientation of the edges of $G$ there exists a path of length at least $(1-\delta)n$
  with at most $\delta n$ edges going backwards.
\end{theorem}

\Cref{thm:Gnp:path} follows from \cref{thm:beta:path} 
by the well-known and simple fact that
for every $\beta>0$ there exists $C>0$ such that $G(n,C/n)$ is \whp{} a $\beta$-graph.
For a proof (of a stronger statement), see e.g.~\cite{GKM22+}*{Lemma 3.4}. 

To prove \cref{thm:beta:path}, we will need the following result from \cite{BKS12}. 
\begin{lemma}[Lemma 4.4 in \cite{BKS12}]\label{lem:DFS}
Let $m,k \geq 1$, let $F$ be a directed graph on $m$ vertices, and suppose that for every pair of disjoint $A,B \subseteq V(F)$ with $|A|,|B| \geq k$, there is an edge of $F$ from $A$ to $B$. Then $F$ has a directed path of length at least $m - 2k + 1$. 
\end{lemma}

\begin{proof}[Proof of \cref{thm:beta:path}]
Fix any orientation of $G$.
Let $\beta=\beta(\delta)$ be small enough (to be chosen later).
Set $\ell := \lfloor \frac{1}{\sqrt{2\beta}}\rfloor - 1$
and $m := \lceil \frac{n}{\ell+2} \rceil$, noting that $m \geq n/(\frac{1}{\sqrt{2\beta}} + 1) = \sqrt{2\beta}n/(1 + \sqrt{2\beta})$ and $m \leq \frac{2n}{\ell+2} \leq 3\sqrt{\beta}n$.
We claim that $G$ contains a collection of $m$ vertex-disjoint directed paths of length $\ell$ each. We find these paths one by one.
Suppose we already found paths $P_1,\dots,P_j$, $j < m$. Let $G'$ be the graph obtained from $G$ by deleting $V(P_1) \cup \dots \cup V(P_j)$. Then $|V(G')| = n - j \cdot (\ell+1) \geq n - (m-1)(\ell+1) \geq \frac{n}{\ell+2}$. Observe that $G'$ has no independent set of size $2\beta n$ because $G$ is a $\beta$-graph. It follows that $\chi(G') \geq \frac{|V(G')|}{2\beta n} \geq \frac{1}{2\beta(\ell+2)} > \ell$, where the last inequality uses that $\ell(\ell+2) < (\ell+1)^2 \leq \frac{1}{2\beta}$. 
So $\chi(G') \geq \ell+1$.
By the Gallai--Hasse--Roy--Vitaver theorem (see e.g.,~\cite{West}*{Theorem~5.1.21}),
every orientation of a graph with chromatic number $k$ contains a directed path of length $k-1$.
Hence, $G'$ contains a directed path of length $\ell$. This proves our claim about the existence of $P_1,\dots,P_m$.
 
For each $1 \leq i \leq m$, let $x_i$ (resp.\ $y_i$) be the first (resp.\ last) vertex of $P_i$. Define an auxiliary directed graph $F$ on $[m]$ in which $i \rightarrow j$ if and only if $\{y_i,x_j\} \in E(G)$. 
We claim that for every pair of disjoint $A,B \subseteq [m]$ with $|A|,|B| \geq \beta n$, there is an edge of $F$ from $A$ to $B$. Indeed, set $X := \{x_i : i \in A\}$ and $Y := \{y_i : i \in B\}$. Since $G$ is a $\beta$-graph, there is an edge in $G$ between $X$ and $Y$; namely there are $j \in A, i \in B$ with $\{y_i,x_j\} \in E(G)$. This gives that $i \rightarrow j$ in $F$.
 
By Lemma \ref{lem:DFS} with $k = \lceil \beta n \rceil$, $F$ has a directed path of length $t := m - 2 \cdot \lceil \beta n \rceil + 1 \geq m-2\beta n - 1$. Without loss of generality, let us assume that this path is $1,\dots,t+1$. Then $\{y_i,x_{i+1}\} \in E(G)$ for every 
 $1 \leq i \leq t$. Now observe that
 \[
    x_1\xrightarrow{P_1} y_1
    \to x_2
    \to \cdots \to
    y_t \to x_{t+1}\xrightarrow{P_{t+1}} y_{t+1}
  \]
is a path of length $(t+1) \cdot \ell + t \geq (t+1)\ell$ with all but at most $t$ edges in the same direction.
We have 
\[
(t+1)\ell \geq (m - 2\beta n) \cdot \left( \frac{1}{\sqrt{2\beta}} - 2 \right) \geq 
\left( \frac{\sqrt{2\beta}n}{1 + \sqrt{2\beta}} - 2\beta n \right) \cdot \left( \frac{1}{\sqrt{2\beta}} - 2 \right) \geq (1 - \delta)n,
\]
provided that $\beta$ is small enough in terms of $\delta$. Also, 
$t \leq m \leq 3\sqrt{\beta}n \leq \delta n$. This completes the proof.  
\end{proof}

The second main tool used in the proof of \cref{thm:Gnp} is the following lemma, which slightly generalises~\cite{GKM22+}*{Lemma 4.1}. 
This lemma allows us to absorb a long path in $G \sim G(n,p)$ into a Hamilton cycle (for $p$ above the Hamiltonicity threshold). The precise statement is as follows:
\begin{lemma}\label{lem:ham:ext}
  For every $\delta>0$ there exists $0<\eps\le\delta$ for which the following holds.
  Let $p=(\log{n}+\log{\log{n}}+\omega(1))/n$ and let $G\sim G(n,p)$.
  Then, \whp{},
  there exists a partition $V(G)=V^\star\cup V'$ with $|V^\star|\le\eps n$
  such that every path $P\subseteq V'$ with $|V(P)|\le (1-\delta)n$ can be extended to a Hamilton cycle in $G$.
\end{lemma}

The proof of \cref{lem:ham:ext} is essentially the same as that of~\cite{GKM22+}*{Lemma 4.1}. For completeness, we give the proof in the appendix. We can now derive \cref{thm:Gnp} from \cref{thm:beta:path} and \cref{lem:ham:ext}.  

\begin{proof}[Proof of \cref{thm:Gnp}]
  Let $\delta>0$ and $p=(\log{n}+\log{\log{n}}+\omega(1))/n$, and let $G\sim G(n,p)$.
  (We assume where needed that $n$ is large enough.)
  By \cref{lem:ham:ext} there exist $0<\eps\le\delta$ and, \whp{}, a partition $V(G)=V^\star\cup V'$ such that $n'=|V'|\ge(1-\eps)n$,
  and any path $P\subseteq V'$ with $|V(P')|\le (1-\delta)n$ can be extended to a Hamilton cycle of $G$.
  Let  $\beta$ be the constant obtained from \cref{thm:beta:path} by plugging in $\delta$.
  As mentioned above, $G$ is \whp{} a $\beta'$-graph for $\beta'=\beta(1-\eps)$.
  If this happens then $G'=G[V']$ is a $\beta$-graph. Now fix any orientation of $G$. 
  By \cref{thm:beta:path} we know that there exists a path $P$ in $G'$ of length at least $(1-\delta)n'$
  (but also at most $(1-\delta)n$)
  with at most $\delta n'$ edges going backwards.
  By the guarantees of \cref{lem:ham:ext} we can, \whp{}, extend $P$ to a Hamilton cycle of $G$,
  having at most $2\delta n'+\eps n\le 3\delta n$ edges going backwards.
  The result follows by taking $\delta\to 0$.
\end{proof}

\section{Proof of \cref{cor:multiplicity}}
\label{sec:tournament}
\begin{proof}[Proof of \cref{cor:multiplicity}]
  Set $\delta'=\delta/2$, and assume for simplicity that $\delta' n$ is an integer.
  Let $\mathcal{P}$ be the set of all paths of length $(1-\delta')n$ in $T$ in which all but at most $\delta' n$ edges are in the same direction.
  Let $C'$ be the constant given by applying \cref{thm:Gnp:path} with $\delta'$. Let $G \sim G(n,C'/n)$.
  The expected number of paths from $\mathcal{P}$ which are present in $G$ is $(C'/n)^{(1-\delta')n} \cdot |\mathcal{P}|$.
On the other hand, at least one such path is present with probability $1 - o(1)$ by \cref{thm:Gnp:path}.
Hence, $(C'/n)^{(1-\delta')n} \cdot |\mathcal{P}| \geq 1/2$,
and so $|\mathcal{P}| \geq 1/2 \cdot (n/C')^{(1-\delta')n}$.
Each path in $\mathcal{P}$ can be extended in $(\delta' n-1)!$ ways into a Hamilton cycle in which at least $(1-\delta)n$ edges are in the same direction,
and each Hamilton cycle is counted at most $n$ times this way.
Hence, the number of such Hamilton cycles is at least
\[
  (\delta'n-1)!\cdot|\mathcal{P}|/n
  \ge \frac{1-o(1)}{2n}\cdot \left(\frac{\delta' n}{3}\right)^{\delta'n} \cdot \left(\frac{n}{C'}\right)^{(1-\delta')n}
  \ge C^{-n}n^n,
\]
for a large enough constant $C$ depending on $C'$ and $\delta$.
\end{proof}

\paragraph{Acknowledgements:} We thank the anonymous referees for a careful reading of the paper and useful comments. 

\bibliography{library}

\begin{appendices}
\newcommand{\CONST}{\textcolor{red}{XXX}}
\section{Proof of \cref{lem:ham:ext}}
As mentioned earlier, the proof is almost identical to the proof of~\cite{GKM22+}*{Lemma 4.1}.
For the proof, we would need the following lemmas and definitions.
At its core, the proof uses P\'osa's rotation--extension technique
(for an overview, see~\cite{Kri16}). 
The following is an immediate consequence of P\'osa's lemma.
\begin{lemma}[P\'{o}sa's lemma~\cite{Pos76}]\label{lem:Posa}
Let $G$ be a graph and let $P = v_0,\dots,v_t$ be a longest path in $G$. Then there exists a set $R \subseteq V(P) \setminus \{v_0\}$ with the following properties:
\begin{enumerate}
    \item For every $v \in R$ there is a path $P'$ with $V(P') = V(P)$ and with endpoints $v_0$ and $v$.
    \item $|N(R)| \leq 2|R|-1$.
\end{enumerate}
\end{lemma}
\begin{proof}
  Take $R$ to be the set of endpoints of paths obtained from $P$ via a sequence of rotations fixing $v_0$.
  P\'osa's lemma \cite{Pos76} (see also \cite{Kri16}) states that $|N(R)| \leq 2|R|-1$.
\end{proof}

Recall that a non-edge
$\{x,y\}$ of $G$ is called a booster if adding $\{x,y\}$ to $G$ creates a graph which is either Hamiltonian or whose longest path is longer than that of $G$.
For a positive integer $k$ and a positive real $\alpha$ we say that a graph $G=(V,E)$ is a \defn{$(k,\alpha)$-expander} if $|N(U)|\ge\alpha|U|$ for every set $U\subseteq V$ of at most $k$ vertices.
The following is a widely-used fact stating that connected $(k,2)$-expanders have many boosters. For a proof, see e.g.~\cite{Kri16}.
\begin{lemma}\label{lem:boosters}
  Let $G$ be a connected $(k,2)$-expander which contains no Hamilton cycle.
  Then $G$ has at least $(k+1)^2/2$ boosters. 
\end{lemma}

Let us now introduce some additional definitions. For a pair of disjoint vertex-sets $U,W$ in a graph, the \defn{density} of $(U,W)$ is defined as 
$d(U,W) := |E(U,W)|/(|U||W|)$. For $\beta,p\in (0,1]$, we say that $G=(V,E)$ is \defn{$(\beta,p)$-pseudorandom} if 
for any two disjoint $U,W\subseteq V$ with $|U|,|W|\ge\beta|V|$ we have $|d(U,W) - p| < \beta p$.
Note that a $(\beta,p)$-pseudorandom graph is, in particular, a $\beta$-graph.

We now move on to establish some useful properties satisfied \whp{} by $G(n,p)$, where $p$ is above the Hamiltonicity threshold.
\begin{lemma}[\cite{GKM22+}*{Lemmas 4.4 and 3.4}]\label{lem:gnp:prop}
  Let $\eps>0$ be sufficiently small, let \linebreak $p=(\log{n}+ \log{\log{n}}+\omega(1))/n$, let $G\sim G(n,p)$
  and let $\beta>0$.
  Then, \whp{},
  \begin{description}[leftmargin=!,labelwidth=\widthof{\bfseries (P1)}]
    \item[\namedlabel{P:min_max_degree}{(P1)}]
      $\delta(G) \geq 2$ and $\Delta(G) \leq 10\log n$.

    \item[\namedlabel{P:smalldist}{(P2)}]
      No vertex $v \in V(G)$ with $d(v) < \log n/10$ is contained in a $3$- or a $4$-cycle, and every two distinct vertices $u,v \in V(G)$ with $d(u),d(v) < \log n/10$ are at distance at least $5$ apart.
    
    \item[\namedlabel{P:sparse_small_sets}{(P3)}]
      Every set $U \subseteq V(G)$ of size at most $\varepsilon n/100$ spans at most $\varepsilon |U| \log n/10$ edges. 
    
    \item[\namedlabel{P:Us}{(P4)}]
      There exist disjoint sets $U_1,U_2 \subseteq V(G)$ with $|U_1|,|U_2|\le \eps n$ for which the following hold for every $v \in V(G)$:
      \begin{description}[leftmargin=!,labelwidth=\widthof{\bfseries (a)}]
        \item[\namedlabel{P:Us:a}{(a)}]
          If $d(v)\ge \log n/10$ then $d(v,U_1),d(v,U_2)\ge \eps \log n/100$;
        \item[\namedlabel{P:Us:b}{(b)}]
          If $d(v)\le \log n/10$ then $v$ and all of its neighbours are in $U_1$.
      \end{description}

    \item[\namedlabel{P:pr}{(P5)}] $G$ is $(\beta,p)$-pseudorandom.
  \end{description} 
\end{lemma}

\begin{lemma}\label{lem:gnp:props_hitting_boosters}
  For every $\zeta>0$ there exists $c>0$ such that the following holds.
  Let $p=(\log{n}+\log{\log{n}}+\omega(1))/n$ and let $G\sim G(n,p)$.
  Then, \whp{}, $G$ satisfies the following:
  for every $W \subseteq V(G)$ of size $|W| \geq\zeta n$
  and for every $(|W|/4,2)$-expander $H$ on $W$ which is a subgraph of $G$ and has at most $c n \log n$ edges,
  if $H$ is not Hamiltonian then $G$ contains a booster with respect to $H$. 
\end{lemma} 
\noindent

\begin{proof}
The proof follows that of~\cite{GKM22+}*{Lemma 4.5}.
  We use a first moment argument. Evidently, the number of choices for the set $W$ is at most $2^n$.
  Let us fix a choice of $W$. For each $t$, the number of choices of $H$ for which $|E(H)| = t$ is at most 
  \[
  \binom{\binom{|W|}{2}}{t} \leq \binom{n^2}{t} \leq 
  \left( \frac{en^2}{t} \right)^t.
  \]
  Now let $H$ be a non-Hamiltonian $(|W|/4,2)$-expander on $W$, and set $t := |E(H)|$.
  It is known and easy to show that if a graph $H$ is a $(|V(H)|/4,2)$-expander then $H$ is connected.
  By \cref{lem:boosters}, $H$ has at least $(|W|/4)^2/2 = |W|^2/32 \geq \zeta^2 n^2/32$ boosters.
  Now, the probability that $G$ contains $H$, but no booster thereof is at most 
  \[
    p^t \cdot (1 - p)^{\zeta^2 n^2/32} \leq 
    p^t \cdot \left( 1 - \frac{\log n}{n} \right)^{\zeta^2 n^2/32} \leq 
    \left( \frac{2\log n}{n} \right)^t \cdot \exp\left( {-\zeta^2 n\log n/32} \right). 
  \]
  Summing over all choices of $W$ and $H$, we see that the 
  probability that the assertion of the lemma does not hold is at most
  \begin{equation}\label{eq:hitting_expander_boosters_union_bound}
    2^n \cdot \exp\left( {-\zeta^2 n\log n/32} \right) \cdot 
    \sum_{t = 1}^{cn\log n}{\left( \frac{2en\log n}{t} \right)^t}. 
  \end{equation}
  Setting $g(t) := (2en\log n/t)^t$, we note that $g'(t) = g(t) \cdot \left( \log(2en\log n/t) - 1 \right) > 0$ for every $t$ in the range of the sum in \eqref{eq:hitting_expander_boosters_union_bound}, assuming $c<1$, say. Thus, this sum is not larger than
  \[
  c n\log n \cdot 
  \left( 2e/c \right)^{c n\log n} = 
  \exp \big( {\big( \log(2e/c) \cdot c + o(1) \big) n \log n} \big).
  \]
  Now, if $c$ is small enough so that $\log(2e/c) \cdot c < \zeta^2/32$,
  we get that \eqref{eq:hitting_expander_boosters_union_bound} tends to $0$ as $n$ tends to infinity. This completes the proof.  
\end{proof}

The following lemma (\cite{GKM22+}*{Lemma 4.6}) states that a graph possessing certain simple properties is necessarily an expander. Such statements are common in the study of Hamiltonicity of random graphs.
\begin{lemma}[\cite{GKM22+}*{Lemma 4.6}]\label{lem:expander_sufficient_conditions}
Let $m,d \geq 1$ be integers and let $H$ be a graph on $h\ge 4m$ vertices satisfying the following properties:
\begin{enumerate}
    \item $\delta(H) \geq 2$;
    \item No vertex $v \in V(H)$ with $d(v) < d$ is contained in a $3$- or a $4$-cycle, and every two distinct vertices $u,v \in V(H)$ with $d(u),d(v) < d$ are at distance at least $5$ apart;
    \item Every set $U \subseteq V(H)$ of size at most $5 m$ contains at most $d|U| / 10$ edges;
    \item There is an edge between every pair of disjoint sets $U_1,U_2 \subseteq V(H)$ of size $m$ each.
\end{enumerate}
Then $H$ is an $(h/4,2)$-expander. 
\end{lemma}

\begin{lemma}\label{lem:ham:Hamiltonian_induced_subgraphs}
  Let $\eps,\zeta > 0$.
  For $p=(\log{n}+\log{\log{n}}+\omega(1))/n$, the random graph $G\sim G(n,p)$ satisfies the following \whp{}.
  Let $W\subseteq V(G)$ be such that $|W| \geq \zeta n$, and for every $v\in W$ it holds that 
  $d(v,W)\ge \min\{ d(v), \eps \log n\}$. Then
  for every $w \in W$ there exists $Y\subseteq W$ with $|Y|\ge \zeta n/4$ such that for each $y\in Y$,
  there is a Hamilton path in $G[W]$ whose endpoints are $w$ and $y$.
\end{lemma}
\noindent

\begin{proof}
    The proof follows that of~\cite{GKM22+}*{Lemma 4.7}.
  We may and will assume $\eps$ is sufficiently small
  (it is enough to have $\eps \leq \min\{1/10,125\zeta,c/2\}$, where $c=c(\zeta)$ is the constant from \cref{lem:gnp:props_hitting_boosters}).
  We will assume that the events defined in \cref{lem:gnp:prop,lem:gnp:props_hitting_boosters} hold, where we use \ref{P:pr} in \cref{lem:gnp:prop} with $\beta=\eps/500$.
  These events hold \whp{}. We will show that then, deterministically, the assertion of \cref{lem:ham:Hamiltonian_induced_subgraphs} holds. 

  It will be convenient to set $d_0 := \eps \log n$.
  Let $W \subseteq V(G)$ be as in the statement of \cref{lem:ham:Hamiltonian_induced_subgraphs}.
  We select a random spanning subgraph $H$ of $G[W]$ as follows. For each $v \in W$, if $d(v,W) < d_0$ then add to $H$ all edges of $G[W]$ incident to $v$. Otherwise, namely if $d(v,W) \geq d_0$, then randomly select a set of $d_0$ edges of $G[W]$ incident to $v$ and add these to $H$. Note that $|E(H)| \leq |W| \cdot d_0 \leq \eps n \log n$. On the other hand, our assumption that $d(v,W)\ge \min\{ d(v), \eps \log n\}$ for every $v \in W$ implies that 
  $\delta(H) \geq \min\{\delta(G), d_0\}$.
  Hence, as $d_0 \geq 2$ (for large enough $n$), we have $\delta(H) \geq 2$ by Property \ref{P:min_max_degree}
    of \cref{lem:gnp:prop}.
  
  We claim that with positive probability (in fact, \whp), $H$ is a $(|W|/4,2)$-expander. 
  In light of \cref{lem:expander_sufficient_conditions},
  it is sufficient to show that with positive probability, $H$ satisfies Conditions 1--4 in that lemma. 
  Here, we will choose the parameters of \cref{lem:expander_sufficient_conditions} as $d := d_0$ and $m := \eps n/500$. Note that $|V(H)| = |W| =\zeta n \geq \varepsilon n/125 = 4m$, as required by \cref{lem:expander_sufficient_conditions}.
  We already showed that $\delta(H) \geq 2$ (which is Condition 1 in \cref{lem:expander_sufficient_conditions}). Let us prove Condition 2. Observe that if $v \in V(H) = W$ has degree less than $d_0$ in $H$, then also $d(v,W) < d_0$ in $G$ (by the definition of $H$). But then, since $d(v,W) \geq \min\{d(v),d_0\}$ by the assumption of \cref{lem:ham:Hamiltonian_induced_subgraphs}, we have $d(v) < d_0 = \varepsilon \log n \leq \log n/10$. 
  Now, by \ref{P:smalldist} in \cref{lem:gnp:prop}, there is no 3- or 4-cycle in $G$ containing a vertex $v$ whose degree in $H$ is less than $d_0$, and every two such vertices are at distance at least 5 apart.   

    Condition 3 in \cref{lem:expander_sufficient_conditions} holds because $H$ is a subgraph of $G$ and due to Property \ref{P:sparse_small_sets} in \cref{lem:gnp:prop} (note that $5m = \varepsilon n /100$). 
    Let us now prove that Condition 4 holds. Let 
  $U_1,U_2 \subseteq V(H) = W$ be disjoint sets satisfying $|U_1|,|U_2| = m = \eps n/500$. Since $G$ is $(\beta,p)$-pseudorandom with $\beta = \varepsilon/500$, we have 
  \begin{equation}\label{eq:sparse_expander_expansion}
  |E_G(U_1,U_2)| \geq 
  (1 - \beta)p \cdot |U_1||U_2| \geq \frac{|U_1||U_2|\log n}{2n} \geq \frac{\eps^2 n \log n}{500000} = \Omega(n\log n)
  \; .
  \end{equation}
  Now, let us bound (from above) the probability that $|E_H(U_1,U_2)| = 0$ (where the randomness is with respect to the choice of $H$).
  Recall that $H$ is defined by choosing, for each $v \in W$, a random set $E(v)$ of $\min\{d(v,W),d_0\}$ edges of $G[W]$ incident to $v$, with all choices made uniformly and independently, and letting $E(H) = \bigcup_{v \in W}{E(v)}$. 
  Fix any $u_1 \in U_1$ with $d(u_1,U_2) \geq 1$, and let $\mathcal{A}_{u_1}$ be the event that 
  there is no edge in $E(u_1)$ with an endpoint in $U_2$. 
  Observe that if 
  $d(u_1,W) < d_0$ then $\pr(\mathcal{A}_{u_1}) = 0$, and otherwise
    \begin{align*}
  \pr(\mathcal{A}_{u_1}) &= \binom{d(u_1,W)-d(u_1,U_2)}{d_0}/\binom{d(u_1,W)}{d_0} = 
  \prod_{i=0}^{d_0-1}{\frac{d(u_1,W) - d(u_1,U_2) - i}{d(u_1,W)-i}} 
  \\&\leq 
  \left(  
  1 - \frac{d(u_1,U_2)}{d(u_1,W)}
  \right)^{d_0} 
  \leq
  \left(  
  1 - \frac{d(u_1,U_2)}{\Delta(G)}
  \right)^{d_0} 
  \leq 
  e^{-d(u_1,U_2) \cdot \frac{d_0}{\Delta(G)}} \leq 
  e^{-\varepsilon d(u_1,U_2)/10} \;.
  \end{align*}
  Here, in the last inequality we used Property \ref{P:min_max_degree} in \cref{lem:gnp:prop}.
  Note that the events \linebreak $(\mathcal{A}_{u_1} : u_1 \in U_1)$ are independent, and that if $E_H(U_1,U_2) = \es$ then $\mathcal{A}_{u_1}$ occurred for every $u_1 \in U_1$ with $d(u_1,U_2) \geq 1$. It now follows that
  \[
  \pr\left( E_H(U_1,U_2) = \es \right) \leq 
  \exp \left( -\frac{\eps}{10} \cdot \sum_{u_1 \in U_1}{d(u_1,U_2)} \right) = 
  \exp \left( -\frac{\eps}{10} \cdot |E_G(U_1,U_2)| \right) \leq 
  e^{-\Omega(n \log n)},
  \]
  where in the last inequality we used \eqref{eq:sparse_expander_expansion}. 
  By taking the union bound over all at most $2^{2n}$ choices of $U_1,U_2$, we see that with high probability, 
  $E_H(U_1,U_2) \neq \es$ for every pair of disjoint sets $U_1,U_2 \subseteq W$ of size $m$ each.
  
  Finally, we apply \cref{lem:expander_sufficient_conditions} to conclude that \whp{} $H$ is a $(|W|/4,2)$-expander.
  From now on, we fix such a choice of $H$.
  Before establishing the assertion of the lemma, we first show that $G[W]$ is Hamiltonian. 
  To find a Hamilton cycle in $G[W]$, we define a sequence of graphs $H_i$, $i \geq 0$, as follows.
  To begin, set $H_0 = H$.
  For each $i \geq 0$, if $H_i$ is Hamiltonian then stop, and otherwise take a booster of $H_i$ contained in $G[W]$ and add it to $H_i$ to obtain $H_{i+1}$.
  That such a booster exists is guaranteed by \cref{lem:gnp:props_hitting_boosters}, as we will always have $|E(H_i)| \leq |E(H)| + |W| \leq |E(H)| + n \leq \eps n\log n + n \leq c/2 \cdot n\log n + n \leq c n \log n$, provided that $\eps$ is smaller than $c/2$, where $c$ is the constant appearing in \cref{lem:gnp:props_hitting_boosters}.
  Note also that $H_i$ is a subgraph of $G[W]$ for each $i \geq 0$. 
  Evidently, this process has to stop (because as long as $H_i$ is not Hamiltonian, the maximum length of a path in $H_i$ is longer than in $H_{i-1}$), thus showing that $G[W]$ must contain a Hamilton cycle, as claimed.
  
  Now let $w \in W$. As $G[W]$ is Hamiltonian, there exists a Hamilton path $P$ of $G[W]$ such that $w$ is one of the endpoints of $P$.
  Evidently, $P$ is a longest path in $G[W]$. Furthermore, note that $G[W]$ is a $(|W|/4,2)$-expander because $H$, a subgraph of $G[W]$, is such an expander. 
  By \cref{lem:Posa} (applied to $G[W]$), there exists a set $R \subseteq V(P) = W$ such that for every $y \in R$ there is a Hamilton path in $G[W]$ with endpoints $w$ and $y$, and such that $|N_{G[W]}(R)| \leq 2|R| - 1$. Now, since $G[W]$ is a $(|W|/4,2)$-expander,
  it must be the case that $|R| > |W|/4 \ge \zeta n/4$.
  So we see that the assertion of the lemma holds with $Y = R$. This completes the proof.  
\end{proof}

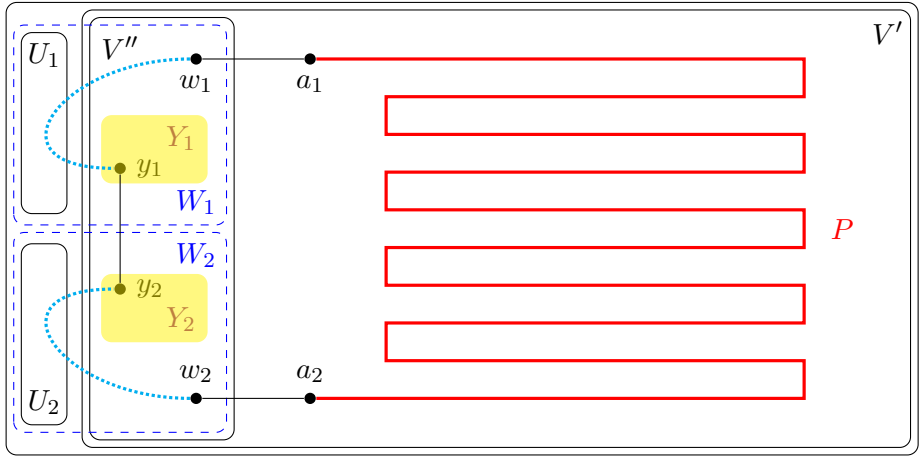
\begin{figure*}[t]
  \captionsetup{width=0.879\textwidth,font=small}
  \centering

\usetikzlibrary{calc}

\tikzset{
  vertex/.style={fill,circle,inner sep=1.5pt},
}

\begin{tikzpicture}
\def\xa{0cm}
\def\xb{1cm}
\def\xba{1.5cm}
\def\xbb{3cm}
\def\xm{6cm}
\def\xc{9cm}
\def\xd{12cm}
\def\xm{6cm}
\def\cC{4.5cm}
\def\xPa{5cm}
\def\xP{10.5cm}

\def\ya{0cm}
\def\yb{3cm}
\def\yc{6cm}
\def\yCa{1.5cm}
\def\yCb{4.5cm}

\def\epl{0.5cm}
\def\eps{0.1cm}

\clip ({\xa-\epl},{\ya-\epl}) rectangle ({\xd+\epl},{\yc+\epl});

\draw[rounded corners] (\xa,\ya) rectangle (\xd,\yc);

\draw[rounded corners]
  ({\xa+2*\eps},{\yb+2*\eps}) rectangle ({\xb-2*\eps},{\yc-4*\eps});
\node (U1) at ({\xa+5*\eps},{\yc-7*\eps}) {$U_1$};

\draw[rounded corners]
  ({\xa+2*\eps},{\ya+4*\eps}) rectangle ({\xb-2*\eps},{\yb-2*\eps});
\node (U2) at ({\xa+5*\eps},{\ya+7*\eps}) {$U_2$};

\draw[rounded corners]
  ({\xb},{\ya+\eps}) rectangle ({\xd-\eps},{\yc-\eps});
\node (V') at ({\xd-4*\eps},{\yc-4*\eps}) {$V'$};

\node[vertex,label={270:$a_1$}] (a1) at (\xPa-2*\epl,\yc-1.5*\epl) {};
\node[vertex,label={90:$a_2$}] (a2) at (\xPa-2*\epl,\ya+1.5*\epl) {};
\draw[very thick,red] (a1)
  -- (\xP,\yc-1.5*\epl)
  -- (\xP,\yc-2.5*\epl)
  -- (\xPa,\yc-2.5*\epl)
  -- (\xPa,\yc-3.5*\epl)
  -- (\xP,\yc-3.5*\epl)
  -- (\xP,\yc-4.5*\epl)
  -- (\xPa,\yc-4.5*\epl)
  -- (\xPa,\yb+0.5*\epl)
  -- (\xP,\yb+0.5*\epl)
  -- (\xP,\yb-0.5*\epl)
  -- (\xPa,\yb-0.5*\epl)
  -- (\xPa,\ya+4.5*\epl)
  -- (\xP,\ya+4.5*\epl)
  -- (\xP,\ya+3.5*\epl)
  -- (\xPa,\ya+3.5*\epl)
  -- (\xPa,\ya+2.5*\epl)
  -- (\xP,\ya+2.5*\epl)
  -- (\xP,\ya+1.5*\epl)
  -- (a2);
\node[red] (P) at (\xP+\epl,\yb) {$P$};

\draw[rounded corners]
  ({\xb+\eps},{\ya+2*\eps}) rectangle ({\xbb},{\yc-2*\eps});
\node (V'') at ({\xb+5*\eps},{\yc-6*\eps}) {$V''$};

\draw[rounded corners,blue,dashed]
  ({\xa+\eps},{\yb+0.5*\eps}) rectangle ({\xbb-\eps},{\yc-3*\eps});
\node[blue] (W1) at ({\xbb-5*\eps},{\yb+3.5*\eps}) {$W_1$};

\draw[rounded corners,blue,dashed]
  ({\xa+\eps},{\ya+3*\eps}) rectangle ({\xbb-\eps},{\yb-0.5*\eps});
\node[blue] (W2) at ({\xbb-5*\eps},{\yb-3.5*\eps}) {$W_2$};

\node[vertex,label={270:$w_1$}] (w1)
  at ({\xbb-1*\epl},{\yc-1.5*\epl}) {};
\draw (a1) -- (w1);

\node[vertex,label={90:$w_2$}] (w2)
  at ({\xbb-1*\epl},{\ya+1.5*\epl}) {};
\draw (a2) -- (w2);

\node (x1)
  at ($(\cC,\yCb)+(-15:{(-3)*\epl} and {1.5*\epl})$) {};

\node[vertex,label={0:$y_1$}] (y1) at (\xba,\yb+8*\eps) {};
\fill[rounded corners,yellow,opacity=0.5]
  ($(y1)-(0.5*\epl,2*\eps)$) rectangle ($(x1)+(-4*\eps,2*\eps)$);
\node[brown] (Y1) at (\xba+3*\eps+1*\epl,\yCb-3*\eps) {$Y_1$};
\draw[very thick,cyan,densely dotted] (w1) to[out=180,in=180,looseness=2.5] (y1);

\node (x2)
  at ($(\cC,\yCa)+(15:{(-3)*\epl} and {1.5*\epl})$) {};

\node[vertex,label={0:$y_2$}] (y2) at (\xba,\yb-8*\eps) {};
\fill[rounded corners,yellow,opacity=0.5]
  ($(y2)-(0.5*\epl,-2*\eps)$) rectangle ($(x2)+(-4*\eps,-2*\eps)$);
\node[brown] (Y2) at (\xba+3*\eps+1*\epl,\yCa+3*\eps) {$Y_2$};
\draw[very thick,cyan,densely dotted] (w2) to[out=180,in=180,looseness=2.5] (y2);

\draw (y1) -- (y2);

\end{tikzpicture}
  \caption{Outline of the proof of \cref{lem:ham:ext}.}
  \label{fig:ham:ext}
\end{figure*}

\begin{proof}[Proof of \cref{lem:ham:ext}]
  Set $\eps=\delta/3$.
  For convenience, we show the existence of a partition $V(G)=V^\star\cup V'$ with $|V^\star|\le 2\eps n$
  instead of $|V^\star|\le\eps n$ (this clearly does not matter).
  We assume that $G$ satisfies the properties detailed in \cref{lem:gnp:prop} for $\beta<\delta/40$ and the assertion of \cref{lem:ham:Hamiltonian_induced_subgraphs}.
  These events happen \whp{}.
  Let $U_1,U_2$ be disjoint subsets of $V=V(G)$ satisfying \ref{P:Us}. 
  Set $V^\star= U_1\cup U_2$ and $V'=V\sm V^\star$,
  and let $P\subseteq V'$ be a path with $|V(P)|\le (1-\delta)n$ and endpoints $a_1,a_2$.
  In particular, $|V^\star|\le 2\eps n$. 
  Our goal is to extend $P$ to a Hamilton cycle of $G$.
  Write $V''=V'\sm V(P)$, 
  partition $V''=V''_1\cup V''_2$ as equally as possible, so $|V''_i| \geq \lfloor (\delta-2\varepsilon)n/2\rfloor \geq \delta n/10$.
  For $i=1,2$, let $W_i=V''_i\cup U_i$ and choose a neighbour $w_i$ of $a_i$ in $W_i$;
  this is possible since $d(a_i,U_i)\ge \eps\log{n}/100$ by \ref{P:Us}.
  Note that $|W_i|\ge \delta n/10$ and for every $v\in W_i$ it holds that $d(v,W_i)\ge d(v,U_i) \geq \min\{ d(v), \eps\log{n}/100\}$ by \ref{P:Us}.
  Hence, by \cref{lem:ham:Hamiltonian_induced_subgraphs}
  (with parameters $\zeta=\delta/10$ and $\varepsilon/100$),
  there exists a set $Y_i\subseteq W_i$ with $|Y_i|\ge \delta n/40$
  such that for every $y\in Y_i$ there is a Hamilton path spanning $W_i$ from $w_i$ to $y$.
  Since $G$ is $(\beta,p)$-pseudorandom for $\beta<\delta/40$,
  there is an edge $e$ in $G$ between $Y_1$ and $Y_2$ with endpoints $y_i\in Y_i$, say.
  For $i=1,2$, denote by $Q_{y_i}$ the Hamilton path between $w_i$ and $y_i$.
  We now construct a Hamilton cycle of $G$ containing $P$ as follows (as depicted in \cref{fig:ham:ext}):
  \[
    a_1 \to w_1 \xrightarrow{Q_{y_1}} y_1 \xrightarrow{e} y_2 \xrightarrow{Q_{y_2}} w_2
    \to a_2 \xrightarrow{P} a_1. \qedhere
  \]
\end{proof}

\end{appendices}

\end{document}